\documentclass[graybox]{svmult}

\usepackage{hyperref}
\usepackage{type1cm}       
\usepackage{makeidx} 
\usepackage{graphicx}
\usepackage{multicol} 
\usepackage[bottom]{footmisc}
\usepackage{newtxtext}        
\usepackage[varvw]{newtxmath}
\usepackage{mathtools}
\usepackage{enumitem}
\usepackage[utf8]{inputenc}
\usepackage[british]{babel}
\def\namedlabel#1#2{\begingroup
    #2%
    \def\@currentlabel{#2}%
    \phantomsection\label{#1}\endgroup
}

\def\div{\operatorname{div}}

\def\RR{\mathbb{R}}

\def\LL{\mathbf{L}}

\def\bsig{\boldsymbol{\sigma}}
\def\Th{\mathcal{T}_h}
\def\Itau{\mathcal{I}_\tau}
\def\I{I}

\def\la{\langle}
\def\ra{\rangle}

\def\Th{\mathcal{T}_h}
\def\Vh{\mathcal{V}_h}
\def\Qh{\mathcal{Q}_h}
\def\Xh{\mathcal{X}_h}

\DeclarePairedDelimiter{\norm}{\|}{\|}
\DeclarePairedDelimiter{\snorm}{|}{|}

\def\E{\mathcal{E}}

\def\u{\mathbf{u}}
\def\vv{\mathbf{v}}
\def\Du{\mathrm{D}\u}

\def\dt{\partial_t}

\def\ddt{\tfrac{d}{dt}}

\usepackage{xcolor}
\def\changesone#1{#1}
\def\changestwo#1{#1}

\makeindex            

\begin{document}

\title*{Structure-Preserving Approximation For The Non-Isothermal Cahn-Hilliard-Navier-Stokes System}
\titlerunning{Approximation For The Non-Isothermal Cahn-Hilliard-Navier-Stokes System}

\author{Aaron Brunk\orcidID{0000-0003-4987-2398} and \\Dennis Schumann\orcidID{ 0009-0006-4853-9947}}
\institute{Aaron Brunk \at Institute of Mathematics Johannes Gutenberg University, Staudinger Weg 9, 55128 Mainz, \\\email{abrunk@uni-mainz.de}
\and Dennis Schumann \at Institute of Mathematics Johannes Gutenberg University, Staudinger Weg 9, 55128 Mainz, \\\email{deschuma@uni-mainz.de}}
\maketitle

\abstract*{In this work we propose and analyse a structure-preserving approximation of the non-isothermal Cahn-Hilliard-Navier-Stokes system using conforming finite elements in space and implicit time discretisation with convex-concave splitting. The system is first reformulated into a variational form that reveals the equations' structure and is then used in the subsequent approximation. }

\abstract{In this work we propose and analyse a structure-preserving approximation of the non-isothermal Cahn-Hilliard-Navier-Stokes system using conforming finite elements in space and implicit time discretisation with convex-concave splitting. The system is first reformulated into a variational form that reveals the equations' structure and is then used in the subsequent approximation. }

\vspace{-1em}
\section{Introduction}
\vspace{-1em}

The non-isothermal Cahn-Hilliard-Navier-Stokes (CHNST) system has increasingly gained attention for the investigation of different phenomena ranging from two-phase flows to fluid-phase-coupled interactions both of which have relevant importance in scientific and industrial applications, such as additive manufacturing and inkjet printing \cite{van2017binary, Yang2020, dadvand2021advected}. For instance, the modelling and simulation of powder bed fusion additive manufacturing (PBF-AM) process employ the non-isothermal CHNST to depict coupled processes in PBF-AM such as fluid-phase interaction, melt flow dynamics and heat transfer \cite{Yang2020}.
For the CHNST system, we consider the following system of partial differential equations
\begin{align}
 \dt\phi &+ \u\cdot\nabla\phi - \div(\LL_{11}\nabla\mu + \LL_{12}\nabla\theta) = 0, \qquad \mu = -\gamma\Delta\phi + \partial_\phi \Psi(\phi,\theta), \label{eq:ac1}\\
 \dt e &+ \u\cdot\nabla e - \div(\LL_{12}\nabla\mu- \LL_{22}\nabla \theta ) - (\eta\Du- \bsig):\nabla\u = 0, \label{eq:ac2} \\
 \dt\u &+ (\u\cdot\nabla)\u - \div(\eta\Du - p\mathbf{I} - \bsig) = 0, \qquad \div(\u)=0. \label{eq:ac3}
\end{align}
The above system is complemented by periodic boundary conditions and initial conditions. In this context $\phi$ denotes a conserved phase-field variable, $\u$ is the flow velocity, $\theta$ is the inverse temperature and $e\equiv e(\phi,\theta)$ is the internal energy. To close the system we consider the Helmholtz free energy density and the Korteweg stress given by 
\begin{align*}
\tilde \Psi(\phi,\theta) :=\Psi(\phi,\theta) + \tfrac{\gamma}{2}\snorm{\nabla\phi}^2, \qquad \bsig:=\tfrac{\gamma}{\theta}\nabla\phi\otimes\nabla\phi.
\end{align*}
From this, one can compute the internal energy and entropy according to \cite{Alt1990,Alt1992}, i.e. $e=\partial_\theta \tilde\Psi$ and $\hat s(e(\theta,\phi),\phi) = s(\theta,\phi)=\theta e-\tilde\Psi$. \changesone{Note that $\Psi(\phi,\theta)$ denotes the gradient-free part of the Helmholtz free energy density. A generic choice can be found in Section 3.}

The above system is already formulated in the inverse temperature $\theta$ which is more convenient for a finite element discretisation. The transformation from the original temperature to the inverse temperature can be found by \changestwo{Pawlow and Alt \cite{Alt1992}, and \cite{brunk2023variational} for more} complicated non-isothermal Cahn-Hilliard-Allen-Cahn system. The general derivation of non-isothermal phase-field models from thermodynamic sound principles is more delicate and we refer to \cite{Charach1998,Fabrizio2006,Pawlow2016} for several different modelling approaches.

For this type of systems \changestwo{in \cite{Guo_2015,Sun2020} the authors consider} finite differences in space together with an Energy Quadratisation (EQ) ansatz. This renders the driving functional quadratic, which facilitates standard time discretisation methods. However, the price of such an ansatz is the relaxation of entropy, which in some cases results in a \emph{drift-off} from the original entropy. 
The isothermal version of the above system, i.e. the Cahn-Hilliard-Navier-Stokes system, is called \emph{model H} introduced by Hohenberg and Halperin \cite{Hohenberg}. This model is employed in the context of two-phase flow, especially in phase separation. In this case, there are many well-established techniques to solve the system with finite elements \cite{Diegel2017,Feng06,Han2015} and finite differences \cite{Gong2018,ChenZhao2020,LiShen2022} in space.

 In this work, we will provide a corresponding fully discrete method for the Cahn-Hilliard-Navier-Stokes system using standard finite elements in space and suitable explicit and implicit time-stepping. The discretisation can be seen as a natural extension of the well-known convex-concave splitting for the isothermal Cahn-Hilliard-Navier-Stokes system to the non-isothermal case.

The work is structured as follows. In Section 2 we will introduce relevant notation, and derive a variational formulation of \changestwo{system} \eqref{eq:ac1}--\eqref{eq:ac3} which is suitable for a finite element approximation. Afterwards, we present the fully discrete method and present the main result, i.e. preservation of total energy conservation and entropy production. Section 3 illustrates the theoretical scheme with a suitable convergence test. In Section 4 we conclude the work and present an outlook for future research.
\vspace{-1em}
\section{Notation And Main Result}
\vspace{-1em}
Before we present \changestwo{our} discretization method and main results in detail, let us briefly introduce our notation and main assumptions, and recall some basic facts.


%
\textbf{Notation:} The system \eqref{eq:ac1} -- \eqref{eq:ac3} is investigated on a finite time interval $(0,T)$ and bounded domain $\Omega$. 
To avoid the discussion of boundary conditions, we consider a spatially periodic setting, i.e., $\Omega \subset \RR^d$, $d=2,3$ is a \changesone{square/cube} and identified with the $d$-dimensional torus $\mathcal{T}^d$. Moreover, functions on $\Omega$ are assumed to be periodic throughout the paper. 
We denote by $\langle \cdot, \cdot\rangle$ the scalar product on $L^2(\Omega)$, which is defined by
\begin{align*}
\la u, v \ra = \int_\Omega u \cdot v \quad \forall u,v \in L^2(\Omega) \quad\text{ with norm } \norm{u}_0^2:=\int_\Omega u^2.    
\end{align*}
We introduce the skew-symmetric formulation of $\mathbf{c}(\u,\vv,\mathbf{w}):=\la(\u\cdot\nabla)\vv,\mathbf{w} \ra$ via
\begin{equation*}
  \mathbf{c}_{skw}(\u,\vv,\mathbf{w}) = \tfrac{1}{2}\mathbf{c}(\u,\vv,\mathbf{w}) - \tfrac{1}{2}\mathbf{c}(\u,\mathbf{w},\vv)  
\end{equation*}
with the relevant property that $\mathbf{c}_{skw}(\u,\vv,\vv)=0$ even if $\u$ is not divergence-free.

In the following proposition, we collect the necessary assumptions applied in the whole work.
\begin{proposition}\label{prop:A}
\changestwo{We require the following assumptions:}
\begin{enumerate}
    \item[] (A1) The interface parameter $\gamma$ is a positive constant. 
    \item[] (A2) The viscosity function $\eta\equiv\eta(\phi,\theta)$ is strictly positive. 
    \item[] (A3) The diffusion matrix $\LL\equiv \LL(\rho,\nabla\rho,\theta)\in\mathbb{R}^{2d \times 2d}$ given by
    $\begin{pmatrix}
        \LL_{11} & -\LL_{12} \\ - \LL_{12} & \LL_{22}
    \end{pmatrix}$ is symmetric and strictly positive definite. \changesone{Note that also $\LL_{12}$ is assumed to be symmetric.} 
    \item[] (A4) For the driving potential $\Psi(\cdot,\cdot):\RR\times\RR_+\to \RR$ we assume that for every fixed $\phi$ the potential $\Psi(\phi,\cdot):\RR_+\to\RR$ is concave and goes to infinity for $\theta\to 0$. Furthermore, for every fixed $\theta\in\RR_+$ the potential $\Psi(\cdot,\theta)$ can be decomposed in a strictly convex and a strictly concave part, denoted by $\Psi_{vex}, \Psi_{cav}.$ 
\end{enumerate}
\end{proposition}

\noindent\textbf{Variational Formulation:}
In this paragraph, we will recast the system into a variational form which is directly usable for a conforming finite element discretisation. We will see that the correct formulation of the Korteweg stress $\bsig$ will be crucial. In the isothermal case it is well-known that a reformulation of $\bsig$ and the pressure $p$ in terms of $\phi,\mu,\pi$, where $\pi$ is a modified pressure, allows us to deduce such variational formulations and we will provide a non-isothermal expansion which reveals the strong coupling of the temperature.

\changestwo{To derive a suitable variational formulation, we assume that a sufficiently regular solution of system \eqref{eq:ac1}--\eqref{eq:ac3} exists at least on a short time interval. Let us fix ideas by assuming a classical solution.}
\begin{definition}\label{def:classical}
We call a solution $(\phi,\theta,u,p)$ of \eqref{eq:ac1}--\eqref{eq:ac3} classical solution, if \eqref{eq:ac1}--\eqref{eq:ac3} holds pointwise for every $(t,x)\in(0,T]\times\Omega$ and the following regularity assumption holds
\begin{align*}
 \phi&\in C^1([0,T],C^4(\bar\Omega)),& \mu&\in C^0([0,T],C^2(\bar\Omega)),\\
 \theta&\in C^1([0,T],C^2(\bar\Omega)),& \u&\in C^1([0,T],C^2(\bar\Omega)),\\
 p&\in C^0([0,T],C^4(\bar\Omega)).  
\end{align*}
\end{definition}

\changestwo{ We note that such solutions might not exist at all or only exist under suitable assumptions for small times.  If the following variational formulation is suitable to prove the existence of generalized solutions is so far open and a topic of future research.}

\begin{lemma}
   Every classical solution \changesone{$(\phi,\mu,\theta,\u,p)$} of the system \eqref{eq:ac1}--\eqref{eq:ac3} fulfills the following variational formulation 
		\begin{align}
		\la \dt\phi,\psi \ra &- \la \phi\u,\nabla\psi\ra + \la \LL_{11}\nabla\mu - \LL_{12}\nabla\theta,\nabla\psi \ra = 0, \label{eq:weak1}\\
		\la \mu,\xi \ra &- \gamma\la \nabla\phi,\nabla\xi \ra - \la \partial_\phi \Psi,\xi\ra = 0, \label{eq:weak2}\\
		\la \dt e,w \ra &+  \la \LL_{12}\nabla\mu - \LL_{22}\nabla\theta, \nabla w\ra  - \la \eta\snorm{\Du}^2, w\ra  \label{eq:weak3}\\
  &- \la \bsig\u,\nabla w \ra- \la \tfrac{\phi}{\theta}\nabla\mu- \boldsymbol{\sigma}\tfrac{\nabla\theta}{\theta},\u w\ra - \la s+\phi\mu,\u\cdot\nabla\tfrac{w}{\theta} \ra  = 0, \notag\\ 
		\la \dt\u,\vv \ra &+ \mathbf{c}_{skw}(\u,\u,\vv) + \la \eta\Du,\mathrm{D}\vv \ra - \la \pi,\div(\vv) \ra \label{eq:weak4} \\
  &+\la \tfrac{\phi}{\theta}\nabla\mu +  (s+\phi\mu)\nabla\tfrac{1}{\theta}-\boldsymbol{\sigma}\tfrac{\nabla\theta}{\theta},\vv \ra = 0, \notag \\
		0 &= \la \div(\u), q \ra \label{eq:weak5}
	\end{align}
		for  $\psi,\xi,w,\vv,q\in C^1(\bar\Omega)$ and  $\pi := p + e - \frac{s+\phi\mu}{\theta}$.
	\end{lemma}
\begin{proof}
    The main part of the procedure is standard, i.e. multiplication by test functions, integration by parts. The Korteweg stress can be expanded as follows
    \begin{align*}
			\div(\bsig)=\div(\tfrac{\gamma}{\theta}\nabla\phi\otimes\nabla\phi) &= \tfrac{\gamma}{\theta}(\Delta\phi\nabla\phi + \nabla\tfrac{1}{2}\snorm{\nabla\phi}^2) + \nabla\phi\otimes\nabla\phi\cdot\nabla\tfrac{\gamma}{\theta} \\
			&= -\tfrac{\mu}{\theta}\nabla\phi + \theta\bsig\nabla\tfrac{1}{\theta} + \tfrac{1}{\theta}(\nabla\tfrac{\gamma}{2}\snorm{\nabla\phi}^2 + \tfrac{\partial\Psi}{\partial\phi}\nabla\phi) , \\
			&= -\tfrac{\mu}{\theta}\nabla\phi + \theta\bsig\nabla\tfrac{1}{\theta} + \tfrac{1}{\theta}(\nabla\tfrac{\gamma}{2}\snorm{\nabla\phi}^2 + \nabla\Psi - e\nabla\theta) \\
			&= -\tfrac{\mu}{\theta}\nabla\phi +  \theta\bsig\nabla\tfrac{1}{\theta}  - \tfrac{1}{\theta}\nabla s + \tfrac{1}{\theta}\theta\nabla e\\
			&= -\tfrac{\mu}{\theta}\nabla\phi + (\theta\bsig + s)\nabla\tfrac{1}{\theta} - \nabla\tfrac{s}{\theta}+ \nabla e \\
            &=  \tfrac{\phi}{\theta}\nabla\mu + \nabla\tfrac{1}{\theta}(s+\phi\mu+\theta\bsig ) - \nabla(\tfrac{s+\phi\mu}{\theta})+ \nabla e. 
		\end{align*}
The Korteweg stress appears in the internal energy equation as well as the Navier-Stokes equation. To derive the weak form \eqref{eq:weak1}-\eqref{eq:weak5} we compute
\begin{align*}
   \la \bsig:\nabla\u,w \ra &= -\la \bsig\u,\nabla w \ra - \la \div(\bsig)\u,w \ra \\
   &= -\la \bsig\u,\nabla w \ra - \la \tfrac{\phi}{\theta}\nabla\mu - \bsig\tfrac{\nabla\theta}{\theta},\u w\ra - \la s+\phi\mu, \u\cdot\nabla\tfrac{w}{\theta}\ra - \la \u\cdot\nabla e,w \ra, \\
   \la \div(\bsig),\vv \ra &= \la \tfrac{\phi}{\theta}\nabla\mu + \nabla\tfrac{1}{\theta}(s+\phi\mu) - \bsig\tfrac{\nabla\theta}{\theta} , \vv\ra - \la e-\tfrac{s+\phi\mu}{\theta}, \div(\vv)\ra.
\end{align*}

Note that after insertion the convective term $\la \u\cdot\nabla e,w \ra$ cancels out.
\end{proof}

\changesone{Note that due to the special structure of the Navier-Stokes equation, we can redefine the pressure i.e. switching from $p$ to $\pi$ as the main variable. For smooth solutions, $p$ and $\pi$ are both regular.}
This variational formulation allows us to deduce the thermodynamics quantities immediately by inserting suitable test functions.
\begin{lemma}\label{lem:cont}
For a classical solution $(\phi,\mu,\theta,\u,\pi)$ of \eqref{eq:ac1} -- \eqref{eq:ac3}, cf. Definition \ref{def:classical}, the conservation of mass and total energy as well as entropy production holds, i.e.
\begin{align*}
 \la \dt\phi,1 \ra &= 0, \qquad \la \dt (\tfrac{1}{2}\snorm{\u}^2+e(\phi,\theta)), 1\ra = 0, \\
 \la \dt s(\phi,\theta),1 \ra &= \norm{\sqrt{\eta\theta}\Du}_0^2 + \la (
    \nabla\mu, \nabla \theta)^\top,\LL (
    \nabla\mu, \nabla \theta)^\top  \ra =: \mathcal{D}_{\theta,\LL}(\Du,\nabla\mu,\nabla\theta)\geq0. 
\end{align*}
\end{lemma}

\begin{proof}
Conservation of mass follows immediately by inserting $\psi=1$. Conservation of total energy is obtain by inserting $w=1$, $\vv=\u,$ $q=\pi$, using the skew-symmetry $\mathbf{c}_{skw}(\u,\u,\u)=0$, i.e.
\begin{align*}
  \ddt&\la \tfrac{1}{2}\snorm{u}^2 + e,1 \ra = \la \dt\u,\u \ra + \la \dt e,1 \ra  \\
  &= -\mathbf{c}_{skw}(\u,\u,\u) - \la \eta\Du,\mathrm{D}\u \ra + \la \pi,\div(\u) \ra \\
  &\quad-\la \tfrac{\phi}{\theta}\nabla\mu +  (s+\phi\mu)\nabla\tfrac{1}{\theta}-\boldsymbol{\sigma}\tfrac{\nabla\theta}{\theta},\u \ra \\
  &\quad+ \la \LL_{12}\nabla\mu - \LL_{22}\nabla\theta, \nabla 1\ra  + \la \eta\snorm{\Du}^2, 1\ra + \la \bsig\u,\nabla 1 \ra \\
  &\quad+ \la \tfrac{\phi}{\theta}\nabla\mu+ \boldsymbol{\sigma}\tfrac{\nabla\theta}{\theta},\u\ra+ \la s+\phi\mu,\u\cdot\nabla\tfrac{1}{\theta} \ra  = 0.
\end{align*}

For the entropy production we insert $\psi=-\mu, \xi=-\dt\phi$ and $w=\theta$ and by cancellation we immediately find
\begin{align}
 \la \dt s,1 \ra & = \la \dt e,\theta \ra - \la \dt\phi, \partial_\phi \Psi \ra - \gamma\la \nabla\phi,\nabla\dt\phi \ra = \la \dt e,\theta \ra - \la \mu,\dt\phi \ra  \label{eq:entropycomp}\\
 &= - \la \phi\u,\nabla\mu\ra + \la \LL_{11}\nabla\mu - \LL_{12}\nabla\theta,\nabla\mu \ra -  \la \LL_{12}\nabla\mu + \LL_{22}\nabla\theta, \nabla \theta\ra \notag   \\
 &+ \la \eta\snorm{\Du}^2, \theta\ra+ \la \bsig\u,\nabla \theta \ra+ \la \tfrac{\phi}{\theta}\nabla\mu- \boldsymbol{\sigma}\tfrac{\nabla\theta}{\theta},\u \theta\ra + \la s+\phi\mu,\u\cdot\nabla 1 \ra \notag\\
 &=  \la \LL_{11}\nabla\mu - \LL_{12}\nabla\theta,\nabla \mu \ra - \la \LL_{12}\nabla\mu-\LL_{22}\nabla\theta, \nabla\theta\ra + \la \eta\snorm{\Du}^2,\theta\ra  \notag\\
& = \norm{\sqrt{\eta\theta}\Du}_0^2 + \la (
    \nabla\mu, \nabla \theta)^\top,\LL (
    \nabla\mu, \nabla \theta)^\top  \ra =: \mathcal{D}_{\theta,\LL}(\Du,\nabla\mu,\nabla\theta)  \geq0. \notag
\end{align}
\end{proof}

\noindent\textbf{Time Discretization:}
We partition the time interval $[0,T]$ \changesone{into uniform} \changestwo{sub-intervals} with step size $\tau>0$ and introduce $\Itau:=\{t^0=0,t^1=\tau,\ldots, t^{n_T}=T\}$, where $n_T=\tfrac{T}{\tau}$ is the absolute number of time steps. We denote by $\Pi^1_c(\Itau;X), \Pi^0(\Itau;X)$ the spaces of continuous piecewise linear and piecewise constant functions on $\Itau$ with values in the space or set $X$. By $g^{n+1},g^n,g^{n+1/2}$ we denote the new, the old time level and the midpoint approximation of $g$, i.e. $(g^{n+1}+g^n)/2$. We introduce the time difference and the discrete time derivative via
\begin{equation*}
    d^{n+1}g = g^{n+1} - g^n, \qquad d^{n+1}_\tau g = \tau^{-1}(g^{n+1}-g^n)=\tau^{-1}d^{n+1}g.
\end{equation*} 
\textbf{Space Discretization:} For \changestwo{the spatial} discretisation we require that $\Th$ is a geometrically conforming partition of $\Omega$ into simplices that can be extended periodically to periodic extensions of $\Omega$. We denote the space of continuous, piecewise linear and quadratic functions over $\Th$ as well as the space of mean free and set of positive piecewise linear functions over $\Th$ via
\begin{align*}
    \Vh &:= \{v \in H^1(\Omega)\cap C^0(\bar\Omega) : v|_K \in P_1(K) \quad \forall K \in \Th\},\\
    \Xh^d &:= \{v \in H^1(\Omega)^d\cap C^0(\bar\Omega)^d : v|_K \in P_2(K)^d \quad \forall K \in \Th\}, \\
    \Qh &:= \{v \in \Vh : \la v, 1\ra=0\},\qquad\Vh^+ := \{v \in \Vh : v(x) > 0,\; \forall x\in\Omega\}.
\end{align*}

We introduce the abbreviation for the convex-concave splitting, \changesone{recall assumption (A4), cf. Proposition \ref{prop:A}} by
\begin{equation*}
 \changesone{\Psi_{sp}}(\phi_h^{n+1},\phi_h^n,\theta_h^{n+1}) := \Psi_{vex}(\phi_h^{n+1},\theta_h^{n+1}) + \Psi_{cav}(\phi_h^{n},\theta_h^{n+1})   
\end{equation*}
and we will abbreviate when suitable $e(\phi_h^{n+1},\theta_h^{n+1})=:e_h^{n+1}$ and similarly for $s,\Psi$.

 We then propose \changestwo{a} fully discrete time-stepping method for the CHNST system.
\begin{problem}\label{prob:ac}
Let \changesone{$(\phi_{h}^0,\u_{h}^0,\theta_{h}^0)\in \Vh\times\Xh^d\times \Vh^+$} be given. Find the functions $(\phi_h,\u_h,\theta_h)\in \Pi^1_c(\Itau;\Vh \times\Xh^d\times \Vh^+)$ and $(\mu_h,\pi_h)\in \Pi^0(\Itau;\Vh\times\Qh)$  such that
\begin{align*}
  \la d_\tau^{n+1}\phi_h,\psi_h \ra &- \la \phi_h^*\u_h^{n+1/2},\nabla\psi_h\ra + \la \LL_{11}^*\nabla\mu_h^{n+1} - \LL_{12}^*\nabla\theta_h^{n+1},\nabla\psi_h \ra = 0, \\
  \la \mu_h^{n+1},\xi_h \ra &- \gamma\la \nabla\phi_h^{n+1},\nabla\xi_h \ra - \la \partial_\phi \changesone{\Psi_{sp}}(\phi^{n+1}_h,\phi^{n}_h,\theta^{n+1}_h),\xi_h \ra = 0, \\
  \la d_\tau^{n+1}e_h,w_h \ra &  - \la \eta^*\snorm{\Du_h^{n+1/2}}^2, w_h\ra + \la \LL_{12}^*\nabla\mu_h^{n+1} - \LL_{22}^*\nabla\theta_h^{n+1}, \nabla w_h\ra\\
  & - \la \bsig_h^*\cdot\u_h^{n+1/2},\nabla w_h\ra-\la \tfrac{\phi_h^*}{\theta_h^{n+1}}\nabla\mu_h^{n+1} - \bsig_h^*\tfrac{\nabla\theta_h^{n+1}}{\theta_h^{n+1}},\u_h^{n+1/2} w_h \ra \\
  &- \la (s_h^*+\phi_h^*\mu_h^*)\u_h^{n+1/2} ,\tfrac{\theta_h^{n+1}\nabla w_h-w_h\nabla\theta_h^{n+1}}{(\theta_h^*)^2} \ra=0, \\
   \la d_\tau^{n+1}\u_h,\vv_h \ra &+ \mathbf{c}_{skw}(\u_h^{*},\u_h^{n+1/2},\vv_h) + \la \eta^*\Du_h^{n+1/2},\mathrm{D}\vv_h \ra - \la \pi_h^{n+1},\div(\vv_h)\ra \\
  & +\la \tfrac{\phi_h^*}{\theta_h^{n+1}}\nabla\mu_h^{n+1} - \bsig_h^*\tfrac{\nabla\theta_h^{n+1}}{\theta_h^{n+1}} - (s_h^*+\phi_h^*\mu^*_h)\tfrac{\nabla\theta_h^{n+1}}{(\theta_h^*)^2},\vv_h \ra=0, \\
 0 &=  \la \div(\u_h^{n+1/2}),q_h \ra
\end{align*}
holds for $(\psi_h,\xi_h,w_h,\vv_h,q_h)\in\Vh\times\Vh\times\Vh^+\times \Xh^d\times\Qh$. Here $g^*$ denotes an evaluation of $g$ at any $t\in\{t^n,t^{n+1}\}$, but all terms have to be evaluated at the same point in time. 
\end{problem}

\begin{theorem}
For any solution $(\phi_h,\mu_h,\u_h,\pi_h,\theta_h)$ of Problem \ref{prob:ac} discrete mass and total energy conservation as well as entropy production holds, i.e.
\begin{align*}
&\la \phi^{n+1}_h -\phi^0_h,1 \ra = 0 ,\qquad \la \tfrac{1}{2}\snorm{\u^{n+1}_h}^2 + e(\phi_h^{n+1},\theta_h^{n+1}) -  \tfrac{1}{2}\snorm{\u^{0}_h}^2 - e(\phi_h^{0},\theta_h^{0}),1 \ra =  0, \\
&\la  s(\phi_h^{n+1},\theta_h^{n+1}) -  s(\phi_h^{0},\theta_h^{0}),1\ra = \tau\sum_{k=0}^{n_T}\mathcal{D}_{\theta^{n+1}_h,\LL^*}(\Du^{n+1}_h,\mu_n^{n+1},\theta_h^{n+1}) + \sum_{k=0}^{n_T}\mathcal{D}_{num}^k,
\end{align*}
where the numerical dissipation satisfies $\mathcal{D}_{num}^n\geq0$ and is given by
\begin{align*}
    \mathcal{D}_{num}^n &= \tfrac{\gamma}{2}\norm{\nabla d^{n+1}\phi_h}^2 - \partial_{\theta\theta}\changesone{\Psi}(\phi^n_h,\xi^3_h)(d^{n+1}\theta_h)^2 \\
    &+ (\partial_{\phi\phi}\changesone{\Psi}_{vex}(\xi^1_h,\theta^{n+1}_h) - \partial_{\phi\phi}\changesone{\Psi}_{cav}(\xi^2_h,\theta^{n+1}_h))(d^{n+1}\phi_h)^2. 
\end{align*}
\changesone{Here $\xi^1_h,\xi^2_h$ are convex combinations of $\phi_h^{n+1},\phi_h^n$ and $\xi^3_h$ is a convex combination of $\theta_h^{n+1},\theta_h^n$.}
\end{theorem}

\begin{proof}
For total energy conservation we insert $\vv_h=\u_h^{n+1/2}, w_h=1$ together with cancellation and the skew-symmetric of $\mathbf{c}_{skew}(\u_h^*,\u_h^{n+1/2},\u_h^{n+1/2})=0$ we obtain
\begin{equation*}
\la \u_h^{n+1}\cdot\u_h^{n+1/2} - \u_h^{n+1/2}\cdot\u_h^n + e_h^{n+1}- e_h^n, 1\ra = 0.
\end{equation*}
Using the algebraic identity $a(a+b) - (a+b)b=a^2-b^2 $ we obtain after rearrangement
\begin{equation*}
    \la \tfrac{1}{2}\snorm{\u^{n+1}_h} + e_h^{n+1} -  \tfrac{1}{2}\snorm{\u^{n}_h}^2 -e_h^n, 1\ra = 0.
\end{equation*}

For the entropy production we compute
\begin{align*}
 \la s_h^{n+1}- s_h^n,1\ra =& \la \theta_h^{n+1}e_h^{n+1} - \theta_h^{n}e_h^n - \changesone{\Psi}_h^{n+1} + \changesone{\Psi}_h^n - \tfrac{\gamma}{2}\snorm{\nabla\phi_h^{n+1}}^2 + \tfrac{\gamma}{2}\snorm{\nabla\phi_h^{n}}^2,1\ra\\
 =&  \la d^{n+1}e_h,\theta_h^{n+1}\ra  - \gamma\la\nabla\phi_h^{n+1},\nabla d^{n+1}\phi_h\ra - \la\changesone{\Psi}_h^{n+1}-\changesone{\Psi}_h^n,1 \ra\\
 &+ \la e_h^n,d^{n+1}\theta_h\ra + \tfrac{\gamma}{2}\snorm{\nabla d^{n+1}\phi_h}^2. 
\end{align*}
Adding $\pm\partial_\phi\changesone{\Psi_{sp}}(\phi_h^{n+1},\phi_h^n,\theta_h^{n+1})d^{n+1}\phi_h$ and insertion of $\xi_h=d^{n+1}\phi_h$ \changestwo{yield}
\begin{align*}
\la s_h^{n+1}- s_h^n,1\ra=& \la d^{n+1} e_h,\theta_h^{n+1}\ra - \la\mu_h^{n+1},d^{n+1}\phi_h\ra - \la \changesone{\Psi}_h^{n+1} - \changesone{\Psi}_h^n - e_h^n d^{n+1}\theta_h,1\ra \\
&+ \la \partial_\phi\Psi(\phi_h^{n+1},\phi_h^n,\theta_h^{n+1}),d^{n+1}\phi_h\ra + \tfrac{\gamma}{2}\norm{\nabla d^{n+1}\phi_h}_0^2  \\
 = & \la d^{n+1} e_h, \theta_h^{n+1}\ra - \la \mu_h^{n+1},d^{n+1}\phi_h \ra + \mathcal{D}_{num}^n.
\end{align*}
Insert $\psi_h=-\tau\mu_h^{n+1}, w_h=\tau\theta_h^{n+1}$ into the discrete formulation \changestwo{we can mimic the same calculations as in the} continuous case for the entropy production equation, cf. proof of Lemma \ref{lem:cont}, cf. \eqref{eq:entropycomp}. \changestwo{Hence, we obtain that
\begin{align*}
\la d^{n+1} e_h, \theta_h^{n+1}\ra - \la \mu_h^{n+1},d^{n+1}\phi_h \ra =  \tau\mathcal{D}_{\theta^{n+1}_h,\LL^*}(\Du^{n+1}_h,\mu_n^{n+1},\theta_h^{n+1})   
\end{align*}}

Finally, for the numerical dissipation we add  $\pm \Psi_h(\phi_h^n,\theta_h^{n+1})$ which yields
\begin{align*}
 \mathcal{D}_{num}^n &= \int_\Omega -\Psi_h^{n+1} +\Psi_h(\phi_h^n,\theta_h^{n+1}) -  e_h^nd^{n+1} \theta_h + \tfrac{\gamma}{2}\snorm{\nabla d^{n+1}\phi_h}^2\\
& - \Psi_h(\phi_h^n,\theta_h^{n+1}) - \Psi_h^n  + \partial_\phi\Psi(\phi_h^{n+1},\phi_h^n,\theta_h^{n+1})d^{n+1}\phi_h \\
& = \int_\Omega \tfrac{\gamma}{2}\norm{\nabla d^{n+1}\phi_h^{n+1}}^2 - \partial_{\theta\theta}\changesone{\Psi}(\phi^n_h,\xi^3_h)(d^{n+1}\theta_h)^2 \\
    &+ (\partial_{\phi\phi}\changesone{\Psi}_{vex}(\xi^1_h,\theta^{n+1}_h) - \partial_{\phi\phi}\changesone{\Psi}_{cav}(\xi^2_h,\theta^{n+1}_h))(d^{n+1}\phi_h^{n+1})^2.
\end{align*}
 Using the structural assumptions on the potential $\Psi$, cf. (A4) in Proposition \ref{prop:A}, we see that $\mathcal{D}_{num}\geq 0$ follows directly.
\end{proof}

\vspace{-1em}
\section{Numerical Test}
\vspace{-1em}
For the convergence test, we set $\Omega=(0,1)^2$ which is identified with the two-torus $\mathbb{T}^2$. This accounts for the periodic boundary conditions. We set $g^*=g^n$ and solve the nonlinear system by the Newton method with tolerance $10^{-12}$ in the $L^2$-norm. 
We consider the initial data
\begin{align*}
 \phi_(x,y) &= 0.4 + 0.2\sin(2\pi x)\sin(2\pi y), \qquad \theta_0(x,y)= 1. + 0.2\sin(2\pi x)\sin(2\pi y)   \\
 \u_0(x,y) &= 10^{-2}(-\sin(\pi x)^2\sin(2\pi y),\sin(2\pi x)\sin(\pi y)^2)
\end{align*}
with the set of functionals and parameters
\begin{align*}
\changesone{\tilde\Psi(\phi,\theta)} &= \log(\theta) + (2\theta-1)\phi^2(1-\phi)^2 + \frac{\gamma}{2}\snorm{\nabla\phi}^2,\\
e &=\tfrac{1}{\theta} + 2\phi^2(1-\phi)^2, \qquad  s = 1-\log(\theta) + \phi^2(1-\phi)^2 - \tfrac{\gamma}{2}\snorm{\nabla\phi}^2, \\
\gamma&=10^{-3},\qquad \eta=10^{-3} + \tfrac{1}{40}(\phi+1)^2,\qquad \LL=10^{-2}\cdot \mathbf{I}.
\end{align*}

Since no analytical solution is available, the discretisation error is estimated by comparing the computed solutions $(\phi_{h,\tau},\mu_{h,\tau},\theta_{h,\tau},\u_{h,\tau},p_{h,\tau})$ with those computed on uniformly refined grids, here halving space and time discretisation parameter i.e. $(\phi_{h/2,\tau/2},\mu_{h/2,\tau/2},\theta_{h/2,\tau/2},\u_{h/2,\tau/2},p_{h/2,\tau/2})$.
The error quantities for the \changestwo{fully discrete} scheme are given in the \changestwo{energy norm}, i.e.
\begin{align*}
	\hspace{-1em}e_{h,\tau} &= \norm*{\phi_{h,\tau} - \phi_{h/2,\tau/2}}_{L^\infty(H^1)}^2 + \norm*{\theta_{h,\tau} - \theta_{h/2,\tau/2}}_{L^\infty(L^2)}^2 + \norm*{\u_{h,\tau} - \u_{h/2,\tau/2}}_{L^\infty(L^2)}^2 \\
	&\,+ \norm*{\mu_{h,\tau} - \mu_{h/2,\tau/2}}_{L^2(H^1)}^2+ \norm*{\theta_{h,\tau} - \theta_{h/2,\tau/2}}_{L^2(H^1)}^2 + \norm*{\u_{h,\tau} - \u_{h/2,\tau/2}}_{L^2(H^1)}^2
\end{align*}
as well as the separated errors, $e^\phi_{h,\tau},e^\mu_{h,\tau},e^{\nabla\theta}_{h,\tau},e^{\nabla\u}_{h,\tau}$, which denote the related single quantities in the above error norm. For the convergence test we choose the discretisation parameters $h_k=2^{-k},\tau_k=10^{-3}\cdot h_k$ for $k=0,\ldots,4$.

\vspace{-1em}
\begin{table}[htbp!]
	\centering
	\small
	\caption{ Errors and experimental orders of convergence for the CHNST system. \label{tab:rates_acnst}} 
	\begin{tabular}{c||c|c|c|c|c|c|c|c|c|c}
		$ k $ & $ e_{h,\tau} $  &  eoc & $e^\phi_{h,\tau}$ & eoc & $e^\mu_{h,\tau}$ & eoc & $e^{\nabla\theta}_{h,\tau}$ & eoc & $e^{\nabla\u}_{h,\tau}$ & eoc  \\
		\hline
		$ 0 $ & $1.42\cdot10^{-0}$ & --- & $1.19 \cdot 10^{-0}$  &   ---   & $1.03 \cdot 10^{-1}$  &   ---    & $9.34 \cdot 10^{-2}$  &   ---  & $1.03\cdot 10^{-2}$ &  --- \\
		$ 1 $ & $5.77\cdot10^{-1}$ & 1.29 &$4.74 \cdot 10^{-1}$  &   1.32  & $5.74 \cdot 10^{-2}$  &   0.85   & $3.84 \cdot 10^{-2}$  & 1.37   & $6.18\cdot 10^{-3}$ & 0.74 \\
		$ 2 $ & $1.59\cdot10^{-1}$ & 1.86 &$1.29 \cdot 10^{-1}$  &   1.88  & $1.69 \cdot 10^{-2}$  &   1.76   & $1.30 \cdot 10^{-2}$  & 1.56   & $4.64\cdot 10^{-4}$ & 3.74 \\
		$ 3 $ & $3.66\cdot10^{-2}$ & 2.12 &$2.88 \cdot 10^{-2}$  &   2.16  & $4.28 \cdot 10^{-3}$  &   1.98   & $3.47 \cdot 10^{-3}$  & 1.90   & $2.66\cdot 10^{-5}$ & 4.13 \\
	\end{tabular}
\end{table}
\vspace{-1em}
We observe in Table \ref{tab:rates_acnst} at least first order convergence, second order for squared norms, in all norms, except for the gradients of the velocity. Indeed, due to the piecewise quadratic elements and the Crank-Nicolson type approximation of the velocity, this can be expected if the coupling terms do not pollute the convergence rate too much. In general, we would however expect that only first order convergence is valid.

\section{Conclusion}

In this work, we have derived the fully discrete finite element scheme for the non-isothermal Cahn-Hilliard-Navier-Stokes system \eqref{eq:ac1} -- \eqref{eq:ac3}. To derive the scheme we have first formulated the continuous equations \changestwo{in} a suitable variational formulation, which allowed for a streamlined Galerkin approximation in space. The time discretisation utilises the usual convex-concave splitting for the Cahn-Hilliard component and the nonlinear implicit time discretisation for the internal energy equation. In future work, we will benchmark the scheme and consider the error analysis. Future extensions to a phase-field dependent density in the spirit of \cite{AGG} will also be considered.

\begin{acknowledgement}
Support by the German Science Foundation (DFG) via SPP~2256: 
 \textit{Variational Methods for Predicting Complex Phenomena in Engineering Structures and Materials} (project BR~7093/1-2) and via TRR~146: \textit{Multiscale Simulation Methods for Soft Matter Systems} (project C3) is gratefully acknowledged. The authors would like to thank D. Trautwein for carefully proofreading the manuscript. We thank 
 the anonymous reviewers for their insightful comments and suggestions.
\end{acknowledgement}

\bibliographystyle{unsrt}
\bibliography{lit.bib}

\end{document}